\documentclass[12pt]{amsart}
\usepackage{graphicx}
\usepackage{hyperref}
\usepackage{xcolor}
\usepackage[english,  activeacute]{babel}
\usepackage[utf8]{inputenc}
\usepackage{amssymb}
\usepackage{amsthm}
\usepackage{graphics,graphicx}
\usepackage{enumerate}
\usepackage{array}
\usepackage{bm}
\usepackage{a4wide}
\usepackage{color, url}
\usepackage{float}
\usepackage{lipsum}
\usepackage{multicol}
\usepackage{mwe}
\usepackage[dvipsnames]{xcolor}
\usepackage{array}
\usepackage{csquotes}
\usepackage[backend=biber]{biblatex}
\usepackage{tikz}
\usetikzlibrary{shapes,arrows,intersections}
\usetikzlibrary{matrix,fit,calc,trees,positioning,arrows,chains,shapes.geometric,shapes,angles,quotes}
\usepackage{nicematrix}


\newtheorem{theorem}{Theorem}[section]
\newtheorem{corollary}{Corollary}[section]
\newtheorem{proposition}{Proposition}[section]
\newtheorem{definition}{Definition}[section]
\newtheorem{lemma}{Lemma}[section]

\newtheorem{example}{Example}{}
\newtheorem{remark}{Remark}{}

\newcommand{\Orb}{\operatorname{Orb}}
\newcommand{\tr}{\operatorname{tr}}
\newcommand{\fix}{\operatorname{Fix}}

\title{On the cycle structure of the symmetric tensor power of permutations}
\author{Sebastian Caballero}
\address{\noindent  Departamento de Matem\'aticas, Universidad de los Andes, Bogot\'a, Colombia}
\email{js.caballerob@uniandes.edu.co}
\author{Diego Villamizar}
\address{\noindent  Department of Mathematics,   Xavier University of Louisiana,  New Orleans, LA 70125}
\email{dvillami@xula.edu}
\urladdr{https://sites.google.com/view/dvillami/}
\date{\today}

\begin{document}

\begin{abstract}
Problem 8.1 in \cite{AstaizaetAl2026} asks about the relationship between the cycle decomposition of a permutation $\sigma$ and that of its symmetric tensor power $\sigma ^{\odot k}$. In this paper, we investigate this question and give formulas for computing the number of fixed points and, in the case of a permutation containing at most one cycle of length greater than one, the number of $s$-cycles.
\end{abstract}
\maketitle

\section{Introduction}
\noindent 

In \cite{AstaizaetAl2026} the authors define a way to produce the symmetric tensor power of a finite graph $G = (V,E)$ on $n \in \mathbb{Z}^{\geq 1}$ vertices by using the symmetric tensor power of the graph's adjacency matrix, defined as a $n\times n$ matrix $A_G$ with entries given by $$[A_G]_{i,j} := \begin{cases}
    1 & \text{ if }v_i,v_j \text{ are adjacent,}\\
    0 & \text{otherwise}.
\end{cases}$$
The authors give a formula to compute the entries of the symmetric tensor power of any matrix $A$, denoted by $A^{\odot k}$, given by Equation \ref{eqEntries} (Lemma 2.2 in \cite{AstaizaetAl2026})
\begin{equation}
\label{eqEntries}
[A^{\odot k}]_{\textbf{i}, \textbf{j}}
= \frac{1}{ \sqrt{ \binom{k}{\textbf{m}(\textbf{i})} \binom{k}{\textbf{m}(\textbf{j})}} }
    \sum_{\substack{ \textbf{p} \in \Orb( \textbf{i}) \\ \textbf{q} \in \Orb( \textbf{j}) }} 
        \prod_{\ell=1}^{k} [A]_{p_{\ell}, q_{\ell}} .
\end{equation}
The formula in Equation \ref{eqEntries} is in terms of $\textbf{i} = (i_1,\cdots ,i_k)$ and $\textbf{j} = (j_1, \cdots , j_k)$, size $n$ compositions of $k$ that we represent here as $k$ tuples with entries in $[n] = \{1,2,\cdots ,n\}$ such that $i_1\leq i_2\leq \cdots \leq i_k$ and $j_1\leq j_2\leq \cdots \leq j_k$ respectively. To denote the set of all size $n$ compositions of $k$ we use $\mathcal{C}_{n,k}$, for example, all size $3$ compositions of $2$ are listed in  $\mathcal{C}_{3,2} = \{11,22,33,12,13,23\}$, when parentheses and commas are omitted. The number of such compositions when $n$ and $k$ are fixed is given by $\left |\mathcal{C}_{n,k}\right | = \binom{k+n-1}{n-1}$ by the Stars and Bars counting technique (see Section 1.2 in \cite{Stanley_2023}).\\\\
Given $\textbf{i}$ a size $n$ composition of $k$, we define $\textbf{m}(\textbf{i})$ to be the vector of multiplicities of each element $1\leq s\leq n$ in the vector $\textbf{i}$. That is $\textbf{m}(\textbf{i}) = (m_1(\textbf{i}),m_2(\textbf{i}),\cdots ,m_n(\textbf{i}))$ where $m_{\ell}(\textbf{i})$ is the number of indices $j \in [k]$ such that $i_j = \ell$. As an example, the values of $\textbf{m}$ in all size $3$ compositions of $2$ are $\textbf{m}(11) = (2,0,0),\, \textbf{m}(22) = (0,2,0),\,\textbf{m}(33) = (0,0,2),\,\textbf{m}(12) = (1,1,0),\,\textbf{m}(13) = (1,0,1),\,\textbf{m}(23) = (0,1,1).$\\

The normalization in Equation \ref{eqEntries} uses the term $\binom{k}{\textbf{m}(\textbf{i})} = \frac{k!}{m_{1}(\textbf{i})!\cdot m_{2}(\textbf{i})! \, \cdots \,  m_{n-1}(\textbf{i})!\cdot m_{n}(\textbf{i})!}$, the multinomial coefficient, and the sum is over sets $\Orb (\textbf{i})$ defined as all tuples  $x \in [n]^k$ that can be constructed by permuting the entries of $\textbf{i}$ in all possible ways. As an example, $\Orb (12) = \{12,21\}$. Notice that the size of $\Orb (\textbf{i}) $ is given by the multinomial coefficient $\binom{k}{\textbf{m}(\textbf{i})}$, this can be seen by a standard counting argument of all possible permutations over permutations that fix the $m_{\ell}(\textbf{i})$ appearances of $\ell$ for all $\ell$'s in $[n]$.\\\\
We will denote the set of permutations of $[n]$ as $\mathfrak{S}_n$ and use the line notation to write down permutations, that is, $\sigma \in \mathfrak{S}_n$ will be written as $\sigma = \sigma(1)\, \sigma(2)\, \cdots \, \sigma (n-1)\, \sigma (n)$. We use the classic notation for writing down the cycle decomposition of a permutation. For example, if $\sigma = 4\,6\,3\,5\,1\,2\,9\,8\,7 \in \mathfrak{S}_9$, then the cycle decomposition is given by $(1\,4\,5)(2\, 6)(3)(7\, 9)(8)$. To a permutation $\sigma \in \mathfrak{S}_n$, we associate a vector $C(\sigma) = (C_1(\sigma),C_2(\sigma),\cdots , C_n(\sigma))$, known as \textit{the cycle type}, where $C_i(\sigma)$ is the number of cycles of length $i$ in $\sigma$. As an example, for $\sigma = 4\,6\,3\,5\,1\,2\,9\,8\,7$ as above, $C\left ( \sigma \right ) = (2,2,1,0,0,0,0,0,0)$ as $
\sigma$ has two cycles of length one, two of length two, and one of length three.\\\\
Problem 8.1 in \cite{AstaizaetAl2026} asks about the relationship between the cycle decomposition of a permutation $\sigma$ and that of its symmetric tensor power $\sigma ^{\odot k}$. In this paper, we investigate this question and give formulas for computing the number of fixed points and, in the case of a permutation containing at most one cycle of length greater than one, we produce a formula to compute the number of $s$-cycles (cycles of length $s$).\\\\
This paper is organized as follows. Section 2 describes the symmetric tensor power of permutations and some basic properties. Section 3 gives combinatorial and algebraic approaches to compute the number of fixed points of the symmetric tensor power of a permutation. Section 4 deals with the number of $s$-cycles of the symmetric tensor power of a permutation that contains at most one cycle of length greater than one.
\section{Symmetric Tensor Power of Permutations} 

One can represent a given permutation $\sigma \in \mathfrak{S}_n$ as a $n \times n$ matrix $A_{\sigma}$ defined by 
\begin{align}
\label{eq:PermMatrix}
\left [A_{\sigma}\right ]_{i,j} := \begin{cases}
    1 & \text{if } \sigma _i = j,\\
    0 & \text{otherwise.} 
\end{cases}    
\end{align}
Left multiplication by $\textbf{v} = (v_1,\cdots ,v_n)$ produces as an image the vector $(v_{\sigma (1)},v_{\sigma (2)},\cdots , v_{\sigma (n)})$. This representation of a permutation by a matrix can be used to define the $k$-th symmetric tensor power of a permutation in the following way.
\begin{definition}
    Let $n\in \mathbb{Z}^{\geq 1}$, $\sigma \in \mathfrak{S}_n$, and $k\geq 1$ a positive integer. We denote by $\sigma ^{\odot k}$ the permutation defined by the matrix $A_{\sigma ^{\odot k}} := A_{\sigma}^{\odot k}$.
\end{definition}
That is, to create the symmetric tensor power, first convert the permutation $\sigma$ into its corresponding permutation matrix $A_{\sigma}$, apply the symmetric tensor power $A_{\sigma}^{\odot k}$ using Equation \ref{eqEntries}. The permutation that defines the symmetric tensor power of $A_{\sigma}$ will be the symmetric tensor power $\sigma ^{\odot k}$ of $\sigma$. The following Proposition (Lemma A.1 in \cite{AstaizaetAl2026}), stated here without a proof, guaranties that $A_{\sigma}^{\odot k}$ is a permutation matrix.
\begin{proposition}[Lemma A.1 in \cite{AstaizaetAl2026}]
\label{prop:diag}
    Let $n\in \mathbb{Z}^{\geq 1}$, $\sigma \in \mathfrak{S}_n$, and $A_{\sigma}$ be the $n\times n$ matrix defined as in Equation \ref{eq:PermMatrix}. For any $k\in \mathbb{Z}^{\geq 1}$, the matrix $A_{\sigma ^{\odot k}}=A_{\sigma}^{\odot k}$ is a permutation matrix of size $\binom{n+k-1}{k} \times \binom{n+k-1}{k}$.
\end{proposition}
\begin{example}
    \label{ex:matrix1}
    Consider $n=3$ and $k=2$. The size $3$ compositions of $2$ are given by $\mathcal{C}_{3,2} = \{11,22,33,12,13,23\}$. In order to be able to produce a matrix, we have to order the compositions in a certain way. Following \cite{AstaizaetAl2026}, we first order the compositions by the number of different elements of $[n]$ that appear in the sequence. For example, $11<12$ because $11$ contains only one element of $[3]$ while $12$ contains two. In the case where two sequences tie, that is, they have the same number of distinct elements, we use the lexicographic order on the sequences assuming the natural order in $[n]$. For example, $12<13$ because even though they have the same number of different elements, namely two, $12$ is smaller lexicographically than $13$.  Let  $3\,1\,2 \in \mathfrak{S}_3$. Using Equation \ref{eq:PermMatrix}, its matrix form is $A_{3\,1\,2 } = \left [ \begin{matrix}
        0&0&1\\
        1&0&0\\
        0&1&0
    \end{matrix}\right ]$. By Proposition \ref{prop:diag}, this is a permutation matrix, and using Equation \ref{eqEntries}, we obtain 
\begin{equation*}
    A_{3\,1\,2}^{\odot 2}= \left [\begin{matrix}
        0&0&1&0&0&0\\
1&0&0&0&0&0\\
0&1&0&0&0&0\\
0&0&0&0&1&0\\
0&0&0&0&0&1\\
0&0&0&1&0&0
    \end{matrix} \right ].
\end{equation*}
 This corresponds to  ${3\,1\,2} ^{\odot 2} = 3\,1\,2\,5\,6\,4$.
\end{example}
Performing the same computation as in Example \ref{ex:matrix1} for each permutation in $\mathfrak{S}_3$, we get all six permutations and their cycle decompositions in Table \ref{tab:symtenpow3}.
\begin{table}[!ht]
    \centering
    \begin{tabular}{c|c|c|c}
    \hline 
    $\sigma \in \mathfrak{S}_3$&Cycle decomp. of $\sigma$&$\sigma ^{\odot 2}$& Cycle decomp. of $\sigma ^{\odot 2}$\\
    \hline 
        $ 1\,2\,3 $&$ (1)(2)(3) $&$ 1\,2\,3\,4\,5\,6 $&$ (1)(2)(3)(4)(5)(6) $\\
$ 1\,3\,2 $&$ (1)(2\, 3) $&$ 1\,3\,2\,5\,4\,6 $&$ (1)(2\,3)(4\,5)(6) $\\
$ 2\,1\,3 $&$ (1\,2)(3) $&$ 2\,1\,3\,4\,6\,5 $&$ (1\,2)(3)(4)(5\,6) $\\
$ 2\,3\,1 $&$ (1\,2\,3) $&$ 2\,3\,1\,6\,4\,5 $&$ (1\,2\,3)(4\,6\,5) $\\
$ 3\,1\,2 $&$ (1\,3\,2) $&$ 3\,1\,2\,5\,6\,4 $&$ (1\,3\,2)(4\,5\,6) $\\
$ 3\,2\,1 $&$ (1\,3)(2) $&$ 3\,2\,1\,6\,5\,4 $&$ (1\,3)(2)(4\,6)(5) $\\
    \hline
    \end{tabular}
    \caption{All Symmetric Tensor $2$-Powers of permutations of  $[3]$.}
    \label{tab:symtenpow3}
\end{table}
\\
It is useful to consider a non-algebraic way to construct $\sigma ^{\odot k}$. A way to do this is by noticing that $\sigma ^{\odot k}$ can be thought of as a permutation of the elements in $\mathcal{C}_{n,k}$ by letting $\sigma$ act diagonally on the composition sequence. This means that if $\textbf{x} = (x_1,\cdots ,x_n) \in \mathcal{C}_{n,k}$, we denote by $\sigma ^{\odot k} (\textbf{x})$ the composition that contains $\sigma (x_1)\sigma(x_2)\cdots \sigma(x_n)$ in its orbit.  The following example uses this method to compute $\sigma ^{\odot 2}$ for $\sigma = 3\, 1\, 2$.

\begin{example}
\label{ex:actDiag}
    Consider $\textbf{x} = 12$, then $\sigma (\textbf{x}) = \sigma(1)\sigma(2) = 31\in \Orb(13)$, where $12, \text{ and }13$ are in $\mathcal{C}_{3,2}$. Table \ref{tab:actionsigma} shows $\sigma ^{\odot 2}(\textbf{x})$ for each $\textbf{x}\in \mathcal{C}_{3,2}$.
    \begin{table}[!ht]
        \centering
        \begin{tabular}{|c|c|}
        \hline
                $\textbf{x}\in \mathcal{C}_{3,2}$&$\sigma ^{\odot 2}(\textbf{x})$\\
                \hline 
             $11$& $\sigma ^{\odot 2}(11)=33$  \\
             $22$& $\sigma ^{\odot 2}(22)=11$\\
             $33$& $\sigma ^{\odot 2}(33)=22$\\
             $12$& $\sigma ^{\odot 2}(12)=31=13$\\
             $13$& $\sigma ^{\odot 2}(13)=32=23$\\
             $23$& $\sigma ^{\odot 2}(23)=12$\\
             \hline
        \end{tabular}
        \caption{The result of applying $\sigma ^{\odot 2}$ to every $\textbf{x}\in \mathcal{C}_{3,2}$}
        \label{tab:actionsigma}
    \end{table}\\
This computation results in $\sigma ^{\odot 2} = 3\,1\,2\,5\,6\,4$ with respect to the chosen order on the compositions. This agrees with Example \ref{ex:matrix1}.
\end{example}
Given a permutation $\sigma \in \mathfrak{S}_n$, one can compute its order, indicated by $\mathrm{ord}(\sigma)$, and defined as the smallest number greater than zero such that $\sigma ^{\operatorname{ord}(\sigma)}:=\underbrace{\sigma \circ \sigma \circ \cdots \circ \sigma}_{\operatorname{ord}(\sigma) \text{ times} }$ produces the identity permutation $\mathtt{Id} = 1\, 2\, \cdots \, n$. The value $\mathrm{ord}(\sigma)$ can be computed using the least common multiple of the lengths of the cycles of $\sigma$. The following Theorem shows that the order is fixed by taking symmetric tensor powers.
\begin{theorem}
\label{theo:order}
    Let $n$ and $k$ be positive integers. Given $\sigma \in \mathfrak{S}_n$, we have $\mathrm{ord}(\sigma) = \mathrm{ord}(\sigma^{\odot k})$.
\end{theorem}
\begin{proof}
    Let $p = \mathrm{ord}(\sigma)$. As shown in Example \ref{ex:actDiag}, we can think of $\sigma^{\odot k}$ as a permutation of elements in $\mathcal{C}_{n,k}$ by letting $\sigma$ act diagonally on the composition sequence; for a composition $i_1i_2\dots i_k$, we have
    \begin{align*}
        (\sigma^{\odot k})^p(i_1 i_2 \cdots i_k) &= \sigma^p(i_1) \sigma^p(i_2) \cdots \sigma^p(i_k)\\
        &= i_1i_2 \cdots i_k.
    \end{align*}
    This implies $\mathrm{ord(\sigma^{\odot k}}) \le p$. Now, suppose that $q:=\mathrm{ord}(\sigma^{\odot k}) < p$, then for each $i \in [n]$, we have $(\sigma^{\odot k})^{q}(i i \cdots i) = \sigma^q(i)\cdots \sigma^q(i) = ii \cdots i$, and hence $\sigma^q(i) = i$ for all $i \in [n]$, implying that $p \le q < p$, a contradiction. Therefore, $\mathrm{ord}(\sigma) = \mathrm{ord}(\sigma^{\odot k})$.
\end{proof}
In the following, we state some basic facts about the complete homogeneous symmetric polynomials (see Section 7.4 in \cite{Stanley_2023}), the trace of the symmetric tensor power, and the eigenvalues of permutation matrices that will be used in Section 3 to compute the number of fixed points of the symmetric tensor power of a permutation. The connection of the fixed points of a permutation $\sigma \in \mathfrak{S}_n$ with the trace of some matrix is that if  $A_{\sigma}$ is defined as in Equation \ref{eq:PermMatrix}, then $\tr (A_{\sigma}) = [A_\sigma]_{1,1}+[A_\sigma]_{2,2}+\cdots [A_\sigma]_{n,n}$ which counts the number of elements $i\in [n]$ such that $\sigma(i)=i$.
\begin{definition}
\label{def:HomPolynomial}
Let $n,k$ be positive integers. The \textit{complete homogeneous symmetric polynomial} in variables $x_1,x_2,\cdots ,x_n$ is defined as
\begin{align}
\label{eq:defHomSymPol}
    h_k(x_1,x_2,\cdots, x_n) := \sum _{a_1+a_2+\cdots +a_n = k}\prod _{i = 1}^nx_i^{a_i},
\end{align}
    where the sum is over all size $n$ compositions of $k$.
\end{definition}
These polynomials are symmetric, permuting the variables leaves the polynomial unchanged, i.e., $h_k(x_1,\cdots, x_n) = h_k(x_{\sigma (1)},\cdots , x_{\sigma(n)})$, for $\sigma \in \mathfrak{S}_n$. 
One way to compute these polynomials $h_k$ is to use the following recursion. 
\begin{lemma}
    \label{lemma:homoSymProp}
    Let $k,n$ be positive integers and $\ell \in [n]$, then
    $$h_k(x_1,\cdots,x_{\ell},\cdots  , x_n) = h_k(x_1,\cdots ,\hat{x}_{\ell},\cdots .\,x_n)+x_{\ell}h_{k-1}(x_1,\cdots , x_{\ell},\cdots ,x_n),$$
    where $\hat{x}_{\ell}$ denotes that the variable is missing.
\end{lemma}
\begin{proof}
    From Equation \ref{eq:defHomSymPol}, we have the following
    \begin{align*}
        h_k(x_1,x_2,\cdots, x_n) &= \sum _{a_1+a_2+\cdots +a_{\ell}+\cdots +a_n = k}x_{\ell}^{a_{\ell}}\prod _{i \neq \ell}x_i^{a_i}.
    \end{align*}
    There are two cases to consider, either $a_{\ell}=0$ or $a_{\ell}>0$. If the former is true, then the sum is over size $n-1$ compositions of $k$ and $x_{\ell}$ does not appear, which yields $h_k(x_1,\cdots ,\hat{x}_{\ell},\cdots ,x_n)$. For the case $a_{\ell}>0$, we can factor out a $x_{\ell}$ and the degree $k$ of each monomial decreases by one, which yields $x_{\ell}h_{k-1}(x_1,\cdots , x_{\ell},\cdots ,x_n)$. Adding the two cases gives the result.
    
\end{proof}
The connection between the complete homogeneous symmetric polynomial and the symmetric tensor power is that the trace of the $k$-th symmetric tensor power corresponds to the complete homogeneous symmetric polynomial evaluated at the eigenvalues of the matrix $A$. This is stated without a proof below.
\begin{proposition}[Corollary 4.1 in \cite{AstaizaetAl2026}]
\label{prop:trace}
    Let $A$ be an $n \times n$ matrix with eigenvalues $\lambda _1,\lambda _2,\cdots , \lambda _n$ and let $k\geq 1$, then $\tr \left (A^{\odot k}\right ) = h_k(\lambda _1,\lambda _2,\cdots ,\lambda _n)$.
\end{proposition}

To use Proposition \ref{prop:trace} when $A$ is a permutation matrix, we must understand the eigenvalues of a permutation matrix. The following proposition tells us that the eigenvalues and their multiplicities are totally determined by the cycle type of the permutation.
    \begin{lemma}
    \label{lemma:eigenvalues}
        Let $\sigma \in \mathfrak{S}_n$ and $\ell \in [n]$. A cycle of length $\ell$ in the decomposition of $\sigma$ contributes once to the multiplicity of the eigenvalue $e^{\frac{2\pi i k}{\ell}}$ for $0 \le k < \ell-1$.
    \end{lemma}
    
    \begin{proof}
        Let $\sigma = (a_1\,a_2\, \dots \, a_\ell)$ be any cycle of length $\ell$, and let $k$ be a non-negative integer $0 \le k < \ell-1$. We explicitly construct an eigenvector for $\lambda = e^{\frac{2 \pi i k}{\ell}}$. Let $\mathbf{v} = (v_1, v_2, \dots, v_n)$ be defined as $\mathbf{v} = \sum _{i = 1}^{\ell}\lambda ^ {i-1}e_{a_{i}}$, where $e_i = (0,\cdots ,1,\cdots ,0)$ denotes the $i$-th element of the standard basis. $\textbf{v}$ contains powers of $\lambda$ exactly at the positions of the elements that belong to the cycle.
        This construction implies that $\lambda$ is an eigenvalue for $A_\sigma$ with eigenvector $\mathbf{v}$ noticing that for $j$ not in the cycle, $(A_\sigma \mathbf{v})_j = 0 = (\lambda \mathbf{v})_j$; for $i \in [\ell-1]$, we have $(A\mathbf{v})_{a_{i}} = v_{a_{i+1}} = (\lambda \mathbf{v})_{a_i}$; and for the case $i = \ell$, we have $(A\mathbf{v})_{a_\ell} = \mathbf{v}_1 = 1 = (\lambda \mathbf{v})_{a_{\ell}}$ 
        therefore, $A\mathbf{v} = \lambda \mathbf{v}$.
        Each cycle of length $\ell$ contributes exactly $\ell$ eigenvalues. All eigenvectors corresponding to different cycles are linearly independent, even if their eigenvalues coincide, so the set of all eigenvectors has size $\sum _{\ell = 1}^n \ell \cdot C_{\ell}(\sigma) = n$. This set is a basis for $\mathbb{C}^n$, ensuring that each cycle contributes exactly once for the multiplicity of each of its eigenvalues.
    \end{proof}
    
    Lemma \ref{lemma:eigenvalues} implies that for $\sigma \in \mathfrak{S}_n$, the multiplicity of $\lambda = 1$ counts the number of cycles in the decomposition of $\sigma$ and the multiplicity of $-1$ counts the number of cycles of even length in the decomposition of $\sigma$.
    \begin{example}
        Let $\sigma = (5\, 1\, 6\, 2)(4\, 3\, 7\, 8\, 10\, 9)\in \mathfrak{S}_{10}$. Using Lemma \ref{lemma:eigenvalues} for the two cycles of lenghts $4$ and $6$ in $\sigma$, the eigenvalues are $e^{\frac{2\pi k}{4}}$ for $k = 0, 1, 2, 3$ and $e^{\frac{2\pi i k}{6}}$ for $k = 0, 1, 2, 3, 4, 5$. By iterating over all possible $k$'s, the eigenvalues are the numbers $1$ and $-1$, each with multiplicity two; and $e^\frac{\pi i}{2}, e^\frac{3\pi i}{2}, e^\frac{\pi i}{3}, e^\frac{2\pi i}{3}, e^\frac{4\pi i}{3}, e^\frac{5\pi i}{3}$ each with multiplicity one.
    \end{example}

\section{The number of Fixed Points of the symmetric tensor power}

    Given a permutation $\sigma \in \mathfrak{S}_n$, the number of fixed points, that is, elements $\ell\in [n]$ such that $\sigma (\ell)=\ell$, is hereafter denoted by $\texttt{Fix}(\sigma)$, i.e., $\texttt{Fix}(\sigma) = \# \{ \ell\in [n]: \sigma (\ell)=\ell \}$. For example, $\texttt{Fix}(21{\color{red}{3}}6{\color{red}{5}}4)=2$. This quantity also corresponds to $C_1(\sigma)$, the number of cycles of length one in the cycle type of $\sigma$. The aim of this section is to give an expression for the number of fixed points of the symmetric tensor power of any permutation $\sigma \in \mathfrak{S}_n$.\\\\
    We will first start with the fixed points of the symmetric tensor power of an involution. An involution is a permutation $\sigma \in \mathfrak{S}_n$ such that $\sigma ^2 = \texttt{id}$. This implies that every cycle of $\sigma$ is of length one (a fixed point) or of length two.  First, we notice that the identity permutation $\mathtt{Id} = 1\,2\, \cdots \, n$ is a trivial involution that fixes all elements, and taking the symmetric tensor power, we get the identity permutation, which implies that $\fix (\mathtt{Id}^{\odot k}) = \binom{n+k-1}{k}$ by means of Proposition \ref{prop:diag}. It is worth noticing that a non-trivial involution is an order two permutation, and so, Theorem \ref{theo:order} implies that the symmetric tensor power is also an order two permutation and hence an involution.\\

\begin{theorem}
    \label{theo:FixedInvoGeneral}
    Let $n$ be a positive integer and $1\leq r\leq \lfloor n/2\rfloor$. Let $\sigma \in \mathfrak{S}_n$ be an involution with cycle decomposition $\sigma = (a_1b_1)\cdots (a_rb_r)$ where all $a_1,a_2,\cdots ,a_r,b_1,b_2,\cdots ,b_r$ are different. The number of fixed points in the symmetric tensor power of $\sigma$ is given by 
    \begin{align}
    \label{eq:FixInvo}
        \fix \left ( \sigma ^ {\odot k}\right ) = \sum _{\ell = 0}^{\lfloor k/2 \rfloor}\binom{\ell +r -1}{\ell}\binom{k-2\ell+n-2r-1}{k-2\ell}.
    \end{align}
\end{theorem}
\begin{proof}
    To construct a fixed point $\textbf{x}\in \mathcal{C}_{n,k}$ of $\sigma ^{\odot k}$, we have to ensure that $\textbf{m}_{a_i}(\textbf{x}) = \textbf{m}_{b_i}(\textbf{x})$ for all $i\in [r]$. Let $q_i = \textbf{m}_{a_i}(\textbf{x})$, by the definition of composition, we have $$2(q_1+\cdots +q_r) +\sum _{i \not \in \{a_1,b_1,a_2,b_2,\cdots ,a_r,b_r\}}\textbf{m}_{i}(\textbf{x}) = k.$$
    
    Let $\ell = q_1+\cdots +q_r$, then $0\leq \ell \leq \lfloor k/2\rfloor$. If $\ell$ is fixed, then we can choose the sequence $\{q_i\}_{i\in [r]}$ in $\binom{\ell +r -1}{\ell}$ ways using a Stars and Bars argument. If we eliminate all $a_i$'s and $b_i$'s from the composition, then we have a size $n-2r$ composition of $k-2\ell$ that can be counted in $\binom{(n-2r)+(k-2\ell)-1}{(n-2r)-1}$. By the multiplication principle,  there are $\binom{\ell +r -1}{\ell}\binom{(n-2r)+(k-2\ell)-1}{(n-2r)-1}$ such compositions. The result follows by adding over all possible $\ell$'s.
\end{proof}
\begin{remark}
\label{rem:invo}
For a particular case of Theorem \ref{theo:FixedInvoGeneral}, when $r=1$, plugging into Equation \ref{eq:FixInvo} we obtain $\displaystyle \fix \left ((ab)^{\odot k}\right ) = \sum _{\ell =0}^{\lfloor k/2\rfloor }\binom{k-2\ell+n-3}{k-2\ell}.$
Instead, we can also compute the number of fixed points using Proposition \ref{prop:trace}. Lemma \ref{lemma:eigenvalues} implies that $\lambda _1 = 1$ is an eigenvalue with multiplicity $n-1$ and $\lambda _2 = -1$ with multiplicity one. Using Proposition \ref{prop:trace}, we obtain
    \begin{align*}
            \fix \left ((ab)^{\odot k}\right ) &= h_k(\underbrace{1,1,\cdots,1}_{n-1},-1)\\
            &=\sum _{a_1+a_2+\cdots +a_{n}=k}1^{a_1+\cdots +a_{n-1}}(-1)^{a_n}\\
            &=\sum _{\ell = 0}^k(-1)^{\ell}\sum _{a_1+\cdots +a_{n-1}+\ell = k}1\\
            &=\sum _{\ell =0 }^k(-1)^{\ell}\binom{n+k-\ell -2}{k-\ell}.
    \end{align*}  
\end{remark}

As noted above, the symmetric tensor power of an involution remains an involution, and so  the permutation $(ab)^{\odot k}\in \mathfrak{S}_n$, as in Remark \ref{rem:invo}, consists only of cycles of length two and fixed points, which allows us to conclude on the number of cycles of length two as follows.
\begin{corollary}
    The number of two cycles on the symmetric tensor power of the permutation $(ab)^{\odot k} \in \mathfrak{S}_n$, where $1\leq a<b\leq n$, is given by
    \begin{align*}
        C_2\left ((ab)^{\odot k}\right) &= \frac{1}{2}\left (\binom{n+k-1}{n-1}-\sum _{\ell =0}^{\lfloor k/2\rfloor }\binom{k-2\ell+n-3}{n-3}\right )
    \end{align*}
\end{corollary}


    
    \noindent
    
In the following, we will determine the number of fixed points of an arbitrary length cycle using Proposition \ref{prop:trace} and Lemma \ref{lemma:eigenvalues}. For a fixed number $\ell >1$, we denote by $f(n, k)$ the number of fixed points of the $k$-th symmetric tensor power of the permutation $(1\, 2\, \dots \, \ell) \in \mathfrak{S}_n$, that is, $f ^{(\ell)}_{n,k} = \fix \left (((1\, 2\, \dots \, \ell)(\ell +1)(\ell +2)\cdots (n)\right )^{\odot k})$. The next theorem describes a recursion for $f ^{(\ell)}_{n,k}$.
    
    \begin{theorem}\label{theorem_recursion_fix}
        Let $\ell>1$ be an integer number. For $k \ge 1$ and $n \ge \ell$ the following recursion is satisfied
        \begin{align*}
        f ^{(\ell)}_{n,k} = \begin{cases}
            n - \ell & \text{if }k = 1,\\
            1 & \text{if } n = \ell \text{ and }k \equiv 0 \pmod \ell,\\
            0 & \text{if } n = \ell \text{ and }k \not\equiv 0  \pmod \ell,\\
            f ^{(\ell)}_{n-1,k}+f ^{(\ell)}_{n,k-1} & \text{otherwise.}
        \end{cases}
        \end{align*}
    \end{theorem}
    
    \begin{proof}
        The case $k=1$ does not change the permutation; therefore, all numbers except those in the $\ell$-cycle are fixed points, and so there are $n-\ell$ fixed points.
        For the case $n=\ell$, consider the trace of the symmetric tensor power of the permutation matrix, by Proposition \ref{prop:trace}, $f ^{(\ell)}_{\ell,k} = h_k(\lambda^0, \lambda^1, \dots, \lambda^{\ell-1})$ where $\lambda = e^{\frac{2\pi i}{\ell}}$. Now, by multiplying all the parameters by $\lambda$, that is, shifting the roots of unity, and using Equation \ref{eq:defHomSymPol}, the definition of $h_k$, the following equation holds $h_k(\lambda \cdot \lambda^0, \lambda \cdot \lambda^1, \dots, \lambda \cdot \lambda^{\ell -1}) = \lambda^{k} h_k(\lambda^0, \lambda^1, \dots, \lambda^{\ell-1}),$
        and by the fact that $h_k$ is symmetric, we can rearrange the parameters to obtain
        $$h_k(\lambda^0, \lambda^1, \dots, \lambda^{\ell -1}) = \lambda^{k} h_k(\lambda^0, \lambda^1, \dots, \lambda^{\ell -1}).$$
        This last equation implies that if $\lambda ^k \neq 1$, then the polynomial must be zero, and hence the number of fixed points $f(\ell,k)=0$ if $k \not \equiv 0 \pmod \ell$. Now suppose that $\ell$ divides $k$, and using Lemma \ref{lemma:homoSymProp}, we conclude that $h_k(\lambda^0, \dots, \lambda^{\ell-1}) = h_k(\lambda^0, \dots, \lambda^{\ell-2}) + \lambda^{\ell -1} h_{k-1}(\lambda^0, \dots, \lambda^{\ell-1}),$
        and since $\ell$ does not divide $k-1$,  $h_{k-1}(\lambda^0, \dots, \lambda^{\ell-1}) = 0$. Therefore, $h_k(\lambda^0, \dots, \lambda^{\ell-1}) = h_k(\lambda^0, \dots, \lambda^{\ell-2})$; using the argument inductively, we conclude that $h_k(\lambda^0, \dots, \lambda^{\ell-1}) = h_k(\lambda^0) = h_k(1) = 1$. For the recursive step, assume that $n>\ell$ and hence the eigenvalue $\lambda = 1$ has multiplicity greater than one, and so $f ^{(\ell)}_{n,k} = h_k(1,1,\lambda _3,\cdots ,\lambda _n)$. Using Lemma \ref{lemma:homoSymProp}, we get 
        \begin{align*}
            f ^{(\ell)}_{n,k} &=   h_k(1,1,\lambda _3,\cdots ,\lambda _n)\\
            &=h_k(1,\lambda _3,\cdots ,\lambda _n)+1\cdot h_{k-1}(1,1,\lambda _3,\cdots ,\lambda _n)\\
            &=f ^{(\ell)}_{n-1,k}+f ^{(\ell)}_{n,k-1}.
        \end{align*}
        The result follows.
    \end{proof}
    \begin{example}
        In the case $\ell = 3$ and  $k = 2$, one can compute the following
        \begin{align*}
            f ^{(3)}_{3,2} &= 0, &f ^{(3)}_{4,2} &= 1,\\
            f ^{(3)}_{5,2} &= 3, &f ^{(3)}_{6,2} &= 6.
        \end{align*}
        Theorem \ref{theorem_recursion_fix} gives the recursion $f ^{(3)}_{n,2} = f ^{(3)}_{n-1,2} + f ^{(3)}_{n,1}$. Solving it yields $f ^{(3)}_{n,2} = \binom{n-2}{2}$.
    \end{example}

\noindent
    \begin{theorem}
    \label{theo:generalFix}
    Given $\sigma \in \mathfrak{S}_n$ having cycle type $C(\sigma )=(c_1,\cdots ,c_n)$. The number of fixed points of $\sigma ^{\odot k}$ is given by
    $$\fix \left (\sigma ^{\odot k} \right )=\sum _ {s_1+2s_2+\cdots +ns_n=k}\prod _{\ell = 1}^n \binom{s_{\ell}+c_{\ell}-1}{s_{\ell}},$$
    where the sum is over sequences of non-negative integers $s_1,s_2,\cdots , s_n$.
\end{theorem}
\begin{proof}
    To construct a fixed point $\mathbf{x} \in \mathcal{C}_{n, k}$ of $\sigma^{\odot k}$, the element $\mathbf{x}$ must satisfy $\mathbf{m}_{i}(\mathbf{x}) = \mathbf{m}_{\sigma(i)}(\mathbf{x})$ for all $i \in [n]$. Consequently, for each cycle $(i_1i_2\dots i_\ell)$ of $\sigma$, all elements $i_j$ must appear in $\mathbf{x}$ the same number of times, that is $\textbf{m}_{i_p}(\textbf{x}) = \textbf{m}_{i_q}(\textbf{x})>0$ for $1\leq p,q\leq \ell$. Let $t_{\ell}$ denote the total number of times an element belonging to an $\ell$-cycles in $\sigma$ appears in $\textbf{x}$, as seen above this number must be divisible by $\ell$, and so let $s_{\ell} = \frac{t_{\ell}}{\ell}$; these numbers must satisfy $s_1 + 2s_2 + \dots+ns_n = k$. For each $\ell$, we must assign the number of times each of the elements of an $\ell$-cycles appear in $\textbf{x}$, so this corresponds to creating a size $c_{\ell}$ (the number of $\ell$-cycles) composition of $s_{\ell}$ . Thus, by the Stars and Bars counting technique, the number of choices is given by $\displaystyle \binom{s_\ell + c_\ell - 1}{s_\ell}$ and by the multiplication principle there are
        $\displaystyle \prod_{\ell = 1}^n \binom{s_{\ell} + c_{\ell} - 1}{s_{\ell}}$ options.
     The result follows by adding all valid choices of $s_1,\cdots ,s_n$.
\end{proof}
\section{Cycle Structure of the Symmetric Tensor Power of a Cycle}
In this section, we extend Theorem \ref{theorem_recursion_fix} to give a formula for the number of $s$-cycles of the symmetric tensor power of the permutation $(1\, 2\, \cdots \, \ell)\in \mathfrak{S}_n$. Let $\ell > 1$ and  $s\leq \ell$ and consider $C^{(\ell,s)}_{n,k}:=C_s((1\,2\, \cdots \, \ell)^{\odot k})$. The case $\ell = 1$ was considered in Section 3 and gives $C^{(1,1)}_{n,k}=\binom{n+k-1}{k}$ and $C^{(1,s)}_{n,k} = 0$ for $s>1$. The following Theorem gives a recursive expression for $C^{(\ell,s)}_{n,k}$.

\begin{theorem} 
\label{theo:recur}
Let $s$ and $\ell$ be integer numbers such that $1\leq s\leq \ell$ and $1<\ell$. For integer numbers $n$ and $k$ such that $\ell\leq n$, the following recursion is satisfied.
\begin{align}
    C^{(\ell,s)}_{n,k} =\begin{cases}
    1 & \text{ if }k=1\text{ and }s=\ell,\\
    n-\ell & \text{ if }k=1\text{ and }s=1,\\
    0 & \text{ if }k=1\text{ and }s\not \in \{1,\ell\},\\
 \displaystyle \left [s\mid \ell \text{ and }\frac{\ell}{s}\mid k\right ]\frac{\ell}{s\ell+sk}\sum _{e|(s,(sk)/\ell)}\mu(e)\binom{(s+sk/\ell)/e}{s/e} & \text{ if }k>1 \text{ and }n=\ell,\\
 C^{(\ell,s)}_{n-1,k}+C^{(\ell,s)}_{n,k-1} & \text{otherwise.}
    \end{cases}
\end{align}
Here, the expression $\left [P\right ] = \begin{cases}1 & \text{ if }P\text{ is true}\\0 & \text{ if }P\text{ is false}\end{cases},$ denotes the \textit{Iverson Bracket}(see Section 2.1 in \cite{graham94}), and $\mu (m)$ is the Möbius multiplicative function(see Section 3.7 in \cite{Stanley_2023}).
\end{theorem}
\begin{proof}
    For $k=1$, we recover the same permutation we started with, and by definition, this contains one cycle of length $\ell$ and $n-\ell$ fixed points. Let $k>1$ and $n>\ell$, and consider a composition $\textbf{x} \in \mathcal{C}_{n,k}$ in a cycle of length $s$. There are two options: $\textbf{m}_n(\textbf{x})=0$ or $\textbf{m}_n(\textbf{x})>0$. If the former is true, then $\textbf{x} \in \mathcal{C}_{n-1,k}$ as it is still a composition of $k$ that belongs to a cycle of length $s$, and if the latter is true, then by eliminating one of the $n$'s from $\textbf{x}$ we can form a new composition $\textbf{y} \in \mathcal{C}_{n,k-1}$, a composition of $k-1$ that belongs to a cycle of length $s$. This gives the recursive step $C^{(\ell,s)}_{n,k} = C^{(\ell,s)}_{n-1,k}+C^{(\ell,s)}_{n,k-1}$. For the case $n=\ell$, notice that there are no cycles of length $s$ if $s$ does not divide $\ell$, the order of the permutation, by Theorem \ref{theo:order}. If $s$ divides $\ell$ and  $\textbf{x} \in \mathcal{C}_{n,k}$ is in a cycle of length $s$, then $\textbf{m}_i(\textbf{x}) = \textbf{m}_{j}(\textbf{x})$ if and only if $i\equiv j \pmod s$, and we have
    \begin{align}
    \label{eq:comp1}
        \frac{\ell}{s}(\textbf{m}_{1}(\textbf{x})+\textbf{m}_{2}(\textbf{x})+\cdots + \textbf{m}_{s}(\textbf{x})) = k,
    \end{align}
     this implies that $\frac{\ell}{s}$ must divide $k$ for this element $\textbf{x}$ to exist. The only case left is if $n=\ell$, $s\mid \ell$ and $\frac{\ell}{s} \mid k$. For this case, consider the following equation 
    \begin{align*}
        \sum _{e  \mid s}e\cdot C^{(\ell,e)}_{\ell,k} = \binom{s+\frac{k}{\ell/s}-1}{s-1}.
    \end{align*}
    The right hand side of this equation is the number of compositions $\textbf{x}\in \mathcal{C}_{n,k}$ that satisfy Equation \ref{eq:comp1}. The left hand side is another way to count these compositions. Notice that these compositions are such that $\sigma ^s (\textbf{x}) = \textbf{x}$ and this implies that $\textbf{x}$ is in a cycle of length $e \in \mathbb{Z}$, a divisor of $s$. Adding the number of compositions over all possible cycles of length $e$, a divisor of $s$, the equation is satisfied. We can condense all these cases when $n=\ell$ in the following equation using the Iverson Bracket.
    
    \begin{align}
    \label{eq:divBin}
        \sum _{e  \mid s}e\cdot C^{(\ell,e)}_{\ell,k} = \left [s\mid \ell \text{ and }\frac{\ell}{s}\mid k\right ]\binom{s+\frac{k}{\ell/s}-1}{s-1}.
    \end{align}

    Applying the Möbius inversion formula (see Section 3.7 in \cite{Stanley_2023}) to Equation \ref{eq:divBin}, we obtain
    \begin{align}
    \label{eq:inverse}
        s\cdot C^{(\ell,s)}_{\ell,k} = \sum _{e \mid s}\mu (e)\left[\frac{s}{e}\mid \ell \text{ and }\frac{\ell}{\frac{s}{e}}\mid k\right ]\binom{s/e+\frac{k}{\ell/(s/e)}-1}{s/e-1}.
    \end{align}
    Notice first that if $\frac{\ell}{\frac{s}{e}}\mid k$ then there exists $q\in \mathbb{Z}$ such that $q\cdot \frac{e\ell}{s}=k$, hence $q\cdot e = \frac{s\cdot k}{\ell}$ and so for the Iverson Bracket to be true we must impose $e \mid \frac{s\cdot k}{\ell}$. Now notice that the binomial coefficient on the right hand side of Equation \ref{eq:inverse} can be written as
    $\binom{s/e+\frac{k}{\ell/(s/e)}-1}{s/e-1} = \frac{s/e}{s/e+\frac{k}{\ell/(s/e)}}\binom{s/e+\frac{k}{\ell/(s/e)}}{s/e} =\frac{\ell}{\ell+k}\binom{\frac{s+sk/\ell}{e}}{s/e}$. Plugging these two observations into Equation $\ref{eq:inverse}$, we get the result.
\end{proof}
\begin{table}[!ht]
        \centering
        
        \begin{tabular}{|c|ccccccccc|}
        \hline
        n/k&1&2&3&4&5&6&7&8&9\\
        \hline
                6 & 1&3&9&20&42&75&132&212&333\\

7 & 1&4&13&33&75&150&282&494&827\\

8 & 1&5&18&51&126&276&558&1052&1879\\

9 & 1&6&24&75&201&477&1035&2087&3966\\

10 & 1&7&31&106&307&784&1819&3906&7872\\

11 & 1&8&39&145&452&1236&3055&6961&14833\\
             \hline
        \end{tabular}
        \caption{The Values of $C^{(6,6)}_{n,k}$ for $\ell=s=6 \leq n<12$ and $1\leq k<10$}
        \label{tab:C2}
    \end{table}
The case $n=s=\ell$ matches row six of sequence A245558 in \cite{oeis} (see Table \ref{tab:C2} for the case $\ell = s=6$) a sequence that was studied in the context of the Hermite reciprocity law (see \cite{Elashvili1999}). It also appears as cases of the number of ordered $k$-subsets of $[n]$ whose sum is congruent to $s$ modulo $n$ (see \cite{broadhurst2025}). Let  $\mathcal{F}_{s,\ell}(x, y) = \sum_{k \ge 1,n \ge \ell} C^{(\ell,s)}_{n,k}x^ny^k$ be the generating function of the sequence $C_s(n,k)$. The following proposition gives an expression for $\mathcal{F}_{s,\ell}(x,y)$.

\begin{corollary}
    Let $s,\ell$ be integer numbers such that $1\leq s\leq \ell$ and $1<\ell$. The bivariate generating function $\mathcal{F}_{s,\ell}(x, y)$ is given in closed form by 
    \begin{align}
    \label{eq:genFunc}
        \mathcal{F}_{s,\ell}(x, y) &= \frac{\mathcal{A}_s(x,y)(1-y) + \mathcal{B}_{s,\ell}(x, y)(1-x) - [s=\ell]x^{\ell}y}{1-x-y},
    \end{align}
    where 
    \begin{align*}
        \mathcal{A}_{s,\ell}(x, y) &= [s=l]x^l y +  \frac{[s |  \ell]}{s} x^\ell \sum_{e | s} \mu(s/e)  \left(\frac{1}{(1-y^{\frac{e\ell}{s}})^e} - \sum_{r=0}^{\lceil2e/\ell\rceil - 1} \binom{e+r-1}{e-1} y^{re\ell/s}\right),\,\\
        \mathcal{B}_{s,\ell}(x, y) &= [s=\ell] x^\ell y + [s=1] \frac{x^{\ell+1}y}{(1-x)^2}.
    \end{align*}
    Moreover, 
    \begin{align}
    \label{eq:ExactForm}
        C^{(\ell,s)}_{n,k}=\displaystyle \frac{[s \mid \ell]}{s}\sum _{e\mid s}\mu(s/e)\sum _{r=0}^{\lfloor\frac{ke}{\ell}\rfloor}\binom{e+r-1}{e-1}\binom{k-\frac{r\ell}{e}+n-\ell-1}{k-\frac{r\ell}{e}}
    \end{align}
\end{corollary}
\begin{proof}
    First, letting $\displaystyle \mathcal{A}_{s,\ell}(x,y)=\sum _{k\geq 1}C^{(\ell,s)}_{\ell,k}x^{\ell}y^k$ and $\displaystyle \mathcal{B}_{s,\ell}(x,y)=\sum _{n\geq \ell}C^{(\ell,s)}_{n,1}x^{n}y^1$, it is clear that the following is a valid decomposition $\mathcal{F}_{s,\ell}(x, y) = \mathcal{A}_{s,\ell}(x, y)+\mathcal{B}_{s,\ell}(x, y)+\sum _{n>\ell , k>1}C^{(\ell,s)}_{n,k}x^ny^k-C^{(\ell,s)}_{\ell,1}x^{\ell}y.$
    The last term involves $C^{(\ell,s)}_{\ell,1}$ which is one exactly when $s=\ell$ and zero otherwise. Using the recursion in Theorem \ref{theo:recur}, we get
    \begin{align*}
\mathcal{F}_{s,\ell}(x, y) &= \mathcal{A}_{s,\ell}(x, y)+\mathcal{B}_{s,\ell}(x, y)+\sum _{\substack{n>\ell\\k>1}}\left (C^{(\ell,s)}_{n-1,k}+C^{(\ell,s)}_{n,k-1}\right )x^ny^k-[s=\ell]x^{\ell}y\\
&= \mathcal{A}_{s,\ell}(x, y)+\mathcal{B}_{s,\ell}(x, y)+\sum _{\substack{n>\ell\\k>1}}C^{(\ell,s)}_{n-1,k}x^ny^k+\sum _{\substack{n>\ell\\k>1}}C^{(\ell,s)}_{n,k-1}x^ny^k-[s=\ell]x^{\ell}y\\
&= \mathcal{A}_{s,\ell}(x, y)+\mathcal{B}_{s,\ell}(x, y)+x\left (\mathcal{F}_{s,\ell}(x,y)-\mathcal{B}_{s,\ell}(x,y)\right )+y\left (\mathcal{F}_{s,\ell}(x,y)-\mathcal{A}_{s,\ell}(x,y)\right )-[s=\ell]x^{\ell}y.
\end{align*}
Solving for $\mathcal{F}_{s,\ell}(x,y)$, Equation \ref{eq:genFunc} follows. Now, consider rewriting Equation \ref{eq:genFunc} as
\begin{align*}
    \mathcal{F}_{s,\ell}(x, y) &= \frac{\mathcal{A}_{s,\ell}(x,y)(1-y)}{(1-y)-x} + \frac{\mathcal{B}_{s,\ell}(x, y)(1-x)}{(1-x)-y} - \frac{[s=\ell]x^{\ell}y}{1-x-y}\\
    &= \frac{\mathcal{A}_{s,\ell}(x,y)}{1-\frac{x}{1-y}} + \frac{\mathcal{B}_{s,\ell}(x, y)}{1-\frac{y}{1-x}} - \frac{[s=\ell]x^{\ell}y}{1-x-y}\\
    &=\mathcal{A}_{s,\ell}(x,y)\sum _{n\geq 0}\frac{x^n}{(1-y)^n}+\mathcal{B}_{s,\ell}(x,y)\sum _{k\geq 0}\frac{y^k}{(1-x)^k}- \frac{[s=\ell]x^{\ell}y}{1-x-y}\\
    &=\underbrace{\sum _{n\geq 0}x^n\mathcal{A}_{s,\ell}(x,y)\sum _{k\geq 0}\binom{k+n-1}{k}y^k}_{*_1}+\underbrace{\sum _{k\geq 0}y^k\mathcal{B}_{s,\ell}(x,y)\sum _{n\geq 0}\binom{k+n-1}{n}x^n}_{*_2}- \frac{[s=\ell]x^{\ell}y}{1-x-y}.
\end{align*}
Plugging $\mathcal{A}_{s,\ell}(x,y)$ into the summation $*_1$ yields
\begin{align*}
    *_1 &=\sum _{n\geq 0}x^n\left (\sum _{k\geq 1}C^{(\ell,s)}_{\ell,k}x^{\ell}y^k\right )\sum _{k\geq 0}\binom{k+n-1}{k}y^k\\
    &=\sum _{n\geq 0}x^nx^{\ell}\sum _{k\geq 1}y^k\sum _{\substack{k_1+k_2=k\\k_1\geq 1,k_2\geq 0}}C^{(\ell,s)}_{\ell,k_1}\binom{k_2+n-1}{k_2}\\
    &=\sum _{n\geq 0}x^{n+\ell}\sum _{k\geq 1}y^k\sum _{k_1=1}^kC^{(\ell,s)}_{\ell,k_1}\binom{k-k_1+n-1}{n-1}\\
    &=\sum _{\substack{n\geq \ell \\ k\geq 1}}x^ny^k\sum _{r=1}^kC^{(\ell,s)}_{\ell,r}\binom{k-r+n-\ell-1}{k-r}.
\end{align*}
Analogously, a similar expression for the summation $*_2$ involving $\mathcal{B}_{s,\ell}(x,y)$, gives 
{\tiny
\begin{align*}
    \mathcal{F}_{s,\ell}(x, y) &= \sum _{\substack{n\geq \ell \\ k\geq 1}}x^ny^k\left (\sum _{r=1}^kC^{(\ell,s)}_{\ell,r}\binom{k-r+n-\ell-1}{k-r}+\sum _{r=\ell}^nC^{(\ell,s)}_{r,1}\binom{k+n-r-2}{n-r}\right )-\frac{[s=\ell]x^{\ell}y}{1-x-y}\\
    &= \sum _{\substack{n\geq \ell \\ k\geq 1}}x^ny^k\left (\sum _{r=1}^kC^{(\ell,s)}_{\ell,r}\binom{k-r+n-\ell-1}{k-r}+\sum _{r=\ell}^nC^{(\ell,s)}_{r,1}\binom{k+n-r-2}{n-r}\right )-[s=\ell]x^{\ell}y\sum _{n\geq 0}(x+y)^n\\
    &= \sum _{\substack{n\geq \ell \\ k\geq 1}}x^ny^k\left (\sum _{r=1}^kC^{(\ell,s)}_{\ell,r}\binom{k-r+n-\ell-1}{k-r}+\sum _{r=\ell}^nC^{(\ell,s)}_{r,1}\binom{k+n-r-2}{n-r}\right )-[s=\ell]\sum_{n\geq 0}\sum_{k\geq 0}\binom{n}{k}y^kx^{n-k}\\
    &= \sum _{\substack{n\geq \ell \\ k\geq 1}}x^ny^k\left (\sum _{r=1}^kC^{(\ell,s)}_{\ell,r}\binom{k-r+n-\ell-1}{k-r}+\sum _{r=\ell}^nC^{(\ell,s)}_{r,1}\binom{k+n-r-2}{n-r}-[s=\ell]\binom{n-\ell+k-1}{k-1}\right ).
\end{align*}}
Taking the coefficient of $x^ny^k$, and using Theorem \ref{theo:recur} for the values of $C_s(\ell,r)$ and $C^{(\ell,s)}_{r,1}=[s=\ell]+[s=1](r-\ell)$, as well as Vandermonde type identities (see Section 5 in \cite{graham94}), we get
\begin{align*}
    C^{(\ell,s)}_{n,k}&=\sum _{r=1}^kC^{(\ell,s)}_{\ell,r}\binom{k-r+n-\ell-1}{k-r}+\sum _{r=\ell}^nC^{(\ell,s)}_{r,1}\binom{k+n-r-2}{n-r}-[s=\ell]\binom{n-\ell+k-1}{k-1}\\
    &=\sum _{r=1}^kC^{(\ell,s)}_{\ell,r}\binom{k-r+n-\ell-1}{k-r}+[s=1]\binom{k+n-\ell-1}{k}\\
    &=\frac{[s \mid \ell]}{s}\sum _{e\mid s}\mu(s/e)\sum _{r=1}^k\left [\frac{\ell}{e}\mid r \right]\binom{e+\frac{er}{\ell}-1}{e-1}\binom{k-r+n-\ell-1}{k-r}+[s=1]\binom{k+n-\ell-1}{k}\\
    &=\frac{[s \mid \ell]}{s}\sum _{e\mid s}\mu(s/e)\sum _{r=1}^{\lfloor\frac{ke}{\ell}\rfloor}\binom{e+r-1}{e-1}\binom{k-\frac{r\ell}{e}+n-\ell-1}{k-\frac{r\ell}{e}}+[s=1]\binom{k+n-\ell-1}{k},
\end{align*}
which completes the proof by noticing that when the inner sum has $r=0$, we obtain $\binom{k+n-\ell-1}{k}$ and $\sum _{e|s}{\mu (e)}=[s=1]$.
\end{proof}
As an example, when $s=1$ and $\ell >1$, using Equation \ref{eq:ExactForm}, we obtain for the sequence $f^{(\ell)}_{n,k}$, defined in Section 3, the expression  $\displaystyle f_{n,k}^{(\ell)} =  \sum _{r=0}^{\lfloor\frac{k}{\ell}\rfloor}\binom{k-{r\ell}+n-\ell-1}{k-{r\ell}}$, which agrees with the expression in Theorem \ref{theo:generalFix}.
\printbibliography

\end{document}